\documentclass{amsart}

\usepackage{amsthm,amsmath,graphicx,epstopdf,comment}

\usepackage{soul,xcolor}

\theoremstyle{plain}
\numberwithin{equation}{section}
\newtheorem{thm}{Theorem}[section]
\newtheorem{theorem}[thm]{Theorem}

\newtheorem{prop}[thm]{Proposition}
\newtheorem{proposition}[thm]{Proposition}

\theoremstyle{definition}

\newtheorem{remark}[thm]{Remark}

\title[Constructing BLFs from handle decompositions]{Constructing broken Lefschetz fibrations from handle decompositions}

\author{Mark C. Hughes}

\begin{document}
\setstcolor{red}

\maketitle

\begin{abstract} 
We present an approach to constructing broken Lefschetz fibrations (BLFs) $f:X\rightarrow S^2$ from a handle decomposition of a 4-manifold $X$.  Given a handle decomposition as input these techniques yield explicit descriptions of the BLFs, and do not rely on classification results from contact topology or the choice of a generic indefinite function like some earlier approaches.  We show that this approach will always work in the case when $X$ is the double of a 4-manifold with boundary, and compute explicit examples of BLFs on connected sums of $S^2$-bundles.  
\end{abstract}

\section{Introduction}

The topic of broken Lefschetz fibrations on smooth 4-manifolds has experienced a great deal of attention in recent years, since their introduction by Auroux, Donaldson, and Katzarkov \cite{AurouxDonaldsonKatzarkov}.  Here, a broken Lefschetz fibration (or BLF) is a smooth surjective maps $f:X\rightarrow \Sigma$ between a smooth compact 4-manifold $X$ and a compact surface $\Sigma$, with only certain types of allowable critical points (see Section~\ref{sec:BLF} for a precise definition).  These singular fibrations share a relationship to near-symplectic structures on 4-manifolds that is analogous to the relationship established by Donaldson and Gompf \cite{Donaldson,GompfandStipsicz} between honest Lefschetz fibrations and symplectic structures.  One key difference, however, is that while Lefschetz pencils exist only on symplectic 4-manifolds, BLFs were also known to exist on smooth 4-manifolds which do not admit any near-symplectic structures.  Indeed, \cite{AurouxDonaldsonKatzarkov} contains the construction of a BLF on $S^4$, which immediately dispelled any hope that the existence of a broken Lefschetz fibration might imply the existence of a near-symplectic structure. 

As there were no examples of 4-manifolds known not to admit any BLFs, the question of whether every smooth, closed 4-manifold admits a BLF was raised \cite{GayandKirby}.  This question was answered in the affirmative independently by Akbulut and Karakurt \cite{AkbulutandKarakurt}, Baykur \cite{Baykur}, and Lekili \cite{Lekili}.  Akbulut and Karakurt's approach built on earlier work by Gay and Kirby \cite{GayandKirby}, and involved cutting the manifold of interest $X$ into two pieces, on which the desired fibration was constructed separately.  Modifications were then made to these fibrations to ensure that they matched along their common boundary.  This process appeals to Giroux's correspondence between open book decompositions of 3-manifolds and contact structures, as well as Eliashberg's classification of overtwisted contact structures.  Baykur and Lekili's approaches each involved modifying a generic indefinite surjective map $X \rightarrow S^2$ near its critical points to obtain a broken Lefschetz fibration $f:X\rightarrow S^2$.

Since each of these general approaches involves either the use of deep classification results from contact topology or the choice of a generic indefinite map, they are not well-suited to dealing with manifolds presented by handlebody descriptions.  Indeed, while a BLF $f:X\rightarrow S^2$ allows one to construct a (non-unique) handle decomposition of the 4-manifold $X$, there are currently no general approaches to constructing an explicit BLF from a given handlebody structure on $X$ (though many explicit examples of BLFs have been found and studied, see e.g. \cite{Baykur2009}).

The purpose of this paper is to present an approach to explicitly constructing BLFs directly from a given handle diagram of a smooth 4-manifold.  These techniques combine the author's work on braided surfaces with caps \cite{Hughes2015} with earlier work by Loi and Piergallini \cite{LoiandPiergallini}.  There the authors construct Lefschetz fibrations on 2-handlebodies (4-manifolds with boundary that have handle decompositions with only 0, 1, and 2-handles), by expressing $X$ as a the total space of a covering $h:X\rightarrow D^4$, branched over a ribbon surface in $D^4$.  The branch locus of this covering is then braided using Rudolph's algorithm \cite{Rudolph1983}, and the resulting map $h':X\rightarrow D^4$ is composed with a suitable projection $\text{pr}:D^4 \rightarrow D^2$ to give a Lefschetz fibration $\text{pr}\circ h' : X \rightarrow D^2$.  These techniques were naturally limited to 2-handlebodies, since higher index handles result in nonribbon branch loci which cannot be braided using Rudolph's algorithm.  

We overcome these difficulties by instead using braided surfaces with caps, which are not required to be ribbon.  Furthermore, because we can specify the boundary of a braided surface with caps, we can combine our construction with techniques of Gay and Kirby to obtain broken Lefschetz fibrations on many closed 4-manifolds.  As an example we show that our approach will produce BLFs on 4-manifolds which are the doubles of 4-manifolds with boundary, and compute explicit examples for connect sums of $S^2$-bundles.

\section{Broken Lefschetz fibrations}
\label{sec:BLF}

This section contains definitions of the various fibrations structures that we will be concerned with on 3 and 4-manifolds, as well as a brief discussions about their topology.

\subsection{Open book decompositions of 3-manifolds}
Let $M$ be a 3-dimensional closed smooth oriented manifold.  An \emph{open book decomposition} on $M$ is a smooth map $\lambda: M \rightarrow D^2$ such that $\lambda^{-1} (\partial D^2)$ is a compact 3-dimensional submanifold on which $\lambda$ restricts as a surface bundle over $S^1 = \partial D^2$.  Furthermore, we require that the closure of $\lambda^{-1}(\text{int }D^2)$ is the disjoint union of solid tori, on which $\lambda$ is the projection $D^2 \times S^1 \rightarrow D^2$.  We say that $\lambda^{-1}(0)$ is the \emph{binding} of the open book on $M$, and for any $p \in S^1$ the compact surface $\Sigma_p = \lambda^{-1}(\{\alpha \cdot p \text{ } | \text{ } 0 \leq \alpha \leq 1 \})$ is the \emph{page} over $p$.  The surface bundle structure on $\lambda^{-1} (\partial D^2)$ induces a monodromy map on the pages of $\lambda$. 

By a celebrated theorem of Giroux \cite{Giroux}, open book decompositions on a closed 3-manifold $M$ (up to a stabilization operation) are in one-to-one correspondence with contact structures on $M$ (up to isotopy).  Thus open book decompositions provide a useful topological setting in which to study contact structures on a given closed 3-manifold.   

\subsection{Singular fibrations on 4-manifolds}
Now let $X$ be a smooth 4-manifold and $\Sigma$ a compact surface, and let $f:X\rightarrow \Sigma$ be a smooth map.  A critical point of $f$ is called a \emph{Lefschetz critical point} if there are orientation preserving local complex coordinates on which $f:\mathbb{C}^2 \rightarrow \mathbb{C}$ is modeled as $f(u,v) = u^2 + v^2$.  If the coordinates around the critical point are instead orientation reversing, then it is called an \emph{anti-Lefschetz critical point}.  

An embedded circle $C \subset X$ of critical points of $f$ is called a \emph{round 1-handle singularity} or \emph{broken singularity} if $f$ is modeled near points of $C$ by the map $(\theta,x,y,z) \mapsto (\theta, x^2+y^2-z^2)$ from $\mathbb{R} \times \mathbb{R}^3 \rightarrow \mathbb{R} \times \mathbb{R}$, where $C$ is given locally by $x=y=z=0$.  

A surjective map $f:X \rightarrow \Sigma$ is called a \emph{Lefschetz fibration} if all critical points of $f$ are in the interior of $X$ and are Lefschetz critical points.  It is called an \emph{achiral Lefschetz fibration} if we also allow anti-Lefschetz critical points.  Finally, we add the adjective \emph{broken} to either of these names to indicate that we also allow round 1-handle singularities in the set of critical points of $f$.  When discussing these maps we will sometimes use the generic term \emph{fibration} to describe a map which can be any of the types defined above.

\subsection{Boundary behavior of fibrations}
\label{sec:BoundaryConditions}

Now suppose that $\partial X \neq \emptyset$ is connected, and that $f:X \rightarrow \Sigma$ is a fibration.  Then we say that $f$ is \emph{convex}, if
\begin{itemize}
\item $\Sigma = D^2$,
\item $f(\partial \Sigma) = D^2$, and
\item $f|_{\partial X} : \partial X \rightarrow D^2$ is an open book decomposition on $\partial X$.
\end{itemize}
We say that $f$ is \emph{concave} if there is a disk $D \subset \text{int }\Sigma$ with 
\begin{itemize}
\item $f(\partial X) = D$, and
\item $f|_{\partial X} : \partial X \rightarrow D$ is an open book decomposition on $\partial X$.
\end{itemize}
Finally, $f$ is said to be \emph{flat} if
\begin{itemize}
\item $f(\partial X) = \partial \Sigma$, and
\item $f|_{\partial X} : \partial X \rightarrow \partial \Sigma$ is a non-singular fiber bundle.
\end{itemize}
The fibers of a flat fibration are all closed surfaces, and the boundary $\partial X$ consists of the fibers  above $\partial \Sigma$.  The fibers of a convex fibration all have boundary, and $\partial X$ is comprised of the fibers above $\partial \Sigma = \partial D^2$, along with the boundaries of the fibers above $\text{int }D^2$.  In contrast, concave fibrations will have both closed fibers and fibers with boundary.  Indeed, the fibers above $\text{int }D \subset \Sigma$ will have boundary, while all other fibers will be closed.

Suppose now that $f_1 : X_1 \rightarrow \Sigma$ is a concave fibration, $f_2: X_2 \rightarrow D^2$ is a convex fibration, and that there is an orientation-reversing diffeomorphism $\phi :\partial X_1 \rightarrow \partial X_2$ which respects the open book decompositions.  Then $f_1$ and $f_2$ can be glued together, to give a fibration $f: X_1 \cup_{\phi} X_2 \rightarrow \Sigma$.  This gives a very useful method for constructing fibrations on closed 4-manifolds.  Indeed, one effective strategy is to divide the closed manifold $X$ into simpler pieces $X_1$ and $X_2$, on which convex and concave fibrations can be constructed.  In general these maps will induce different open book decompositions along their common boundary.  If, however, these fibrations can be modified so that they agree along $\partial X_1 = \partial X_2$, then they can be glued to give a fibration on all of $X$.  See \cite{AkbulutandKarakurt,EtnyreandFuller,GayandKirby} for approaches to matching these boundary fibrations which make use of Giroux's theorem and Eliashberg's classification of overtwisted contact structures.  In the case of closed 4-manifolds our approach will also involve splitting $X$ into two pieces, though we construct the convex fibration from the boundary inwards, so that it can be made to match a given concave fibration along its boundary.

\subsection{The topology of broken fibrations}
\label{sec:TopologyOfBrokenFibrations} 

The regular fibers of a flat or convex (achiral) Lefschetz fibration $f : X \rightarrow \Sigma$ will all be surfaces of the same diffeomorphism type, which we call the \emph{genus of} $f$.  Lefschetz fibrations of genus $g \geq 2$ can be determined entirely by their \emph{monodromy representations}.  Let $\Sigma^* \subset \Sigma$ denote the set of regular values of $f$, and let $p \in \Sigma \backslash \Sigma^*$ be a critical value.  If $\gamma \subset \Sigma^*$ is an oriented loop based at $q \in \Sigma^*$ which travels counterclockwise around $p$ and no other critical values, then a trivialization of the bundle over $\gamma$ induces a diffeomorphism of the fiber $F_q$ above $q$.  This diffeomorphism will be a right-handed (left-handed) Dehn twist if $p$ corresponds to a Lefschetz critical point (anti-Lefschetz critical point respectively).  The cycle along which this Dehn twist takes place is called the \emph{vanishing cycle} associated to the critical point.  As we approach the critical fiber $F_p$, the corresponding vanishing cycles in nearby regular fibers shrink down to a single transverse intersection in $F_p$ (see Figure~\ref{fig:LefschetzCriticalPoints} where the vanishing cycle is denoted with a dashed line).

\begin{figure}
 \centering
 \includegraphics[width=0.4\textwidth]{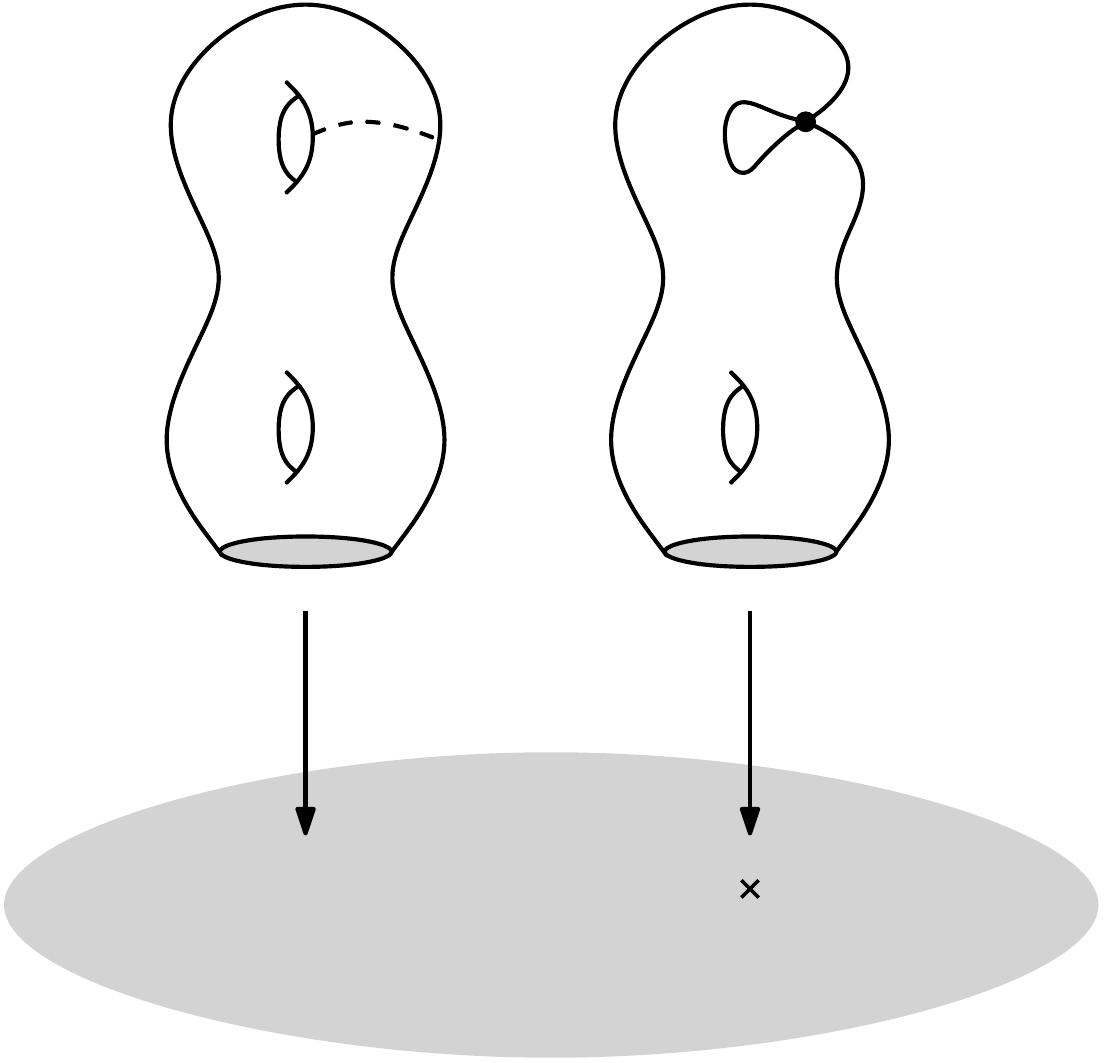}
  \caption{Vanishing cycle of Lefschetz critical point}
  \label{fig:LefschetzCriticalPoints}
\end{figure}

Now suppose that $f : X \rightarrow \Sigma$ is a broken fibration, with round 1-handle singularity along an embedded circle $C$.  Suppose that $C' \subset \Sigma$ is the image of $C$ under $f$, and that $C'$ is embedded.  Let $p$ and $q$ be nearby regular points sitting on opposite sides of $C'$.  Suppose for concreteness that $p=(\theta,-1)$ and $q=(\theta,1)$ for some $\theta \in S^1$ in the coordinate charts described above.  Then the fiber $F_q$ above $q$ can be obtained from $F_p$ by 0-surgery along a pair of points in $F_p$.  Equivalently, $F_p$ can be obtained from $F_q$ by 1-surgery along a simple closed curve (see Figure~\ref{fig:BrokenCriticalPoint}).  Indeed, we can think of the coordinate charts describing the round 1-handle singularity as defining an $S^1$ family of local Morse functions, each with a single index 1 critical point.  In particular, the genus of the fiber of a broken fibration changes by $\pm 1$ each time we cross the image of a round 1-handle singularity in $\Sigma$.

\begin{figure}
 \centering
 \includegraphics[width=0.4\textwidth]{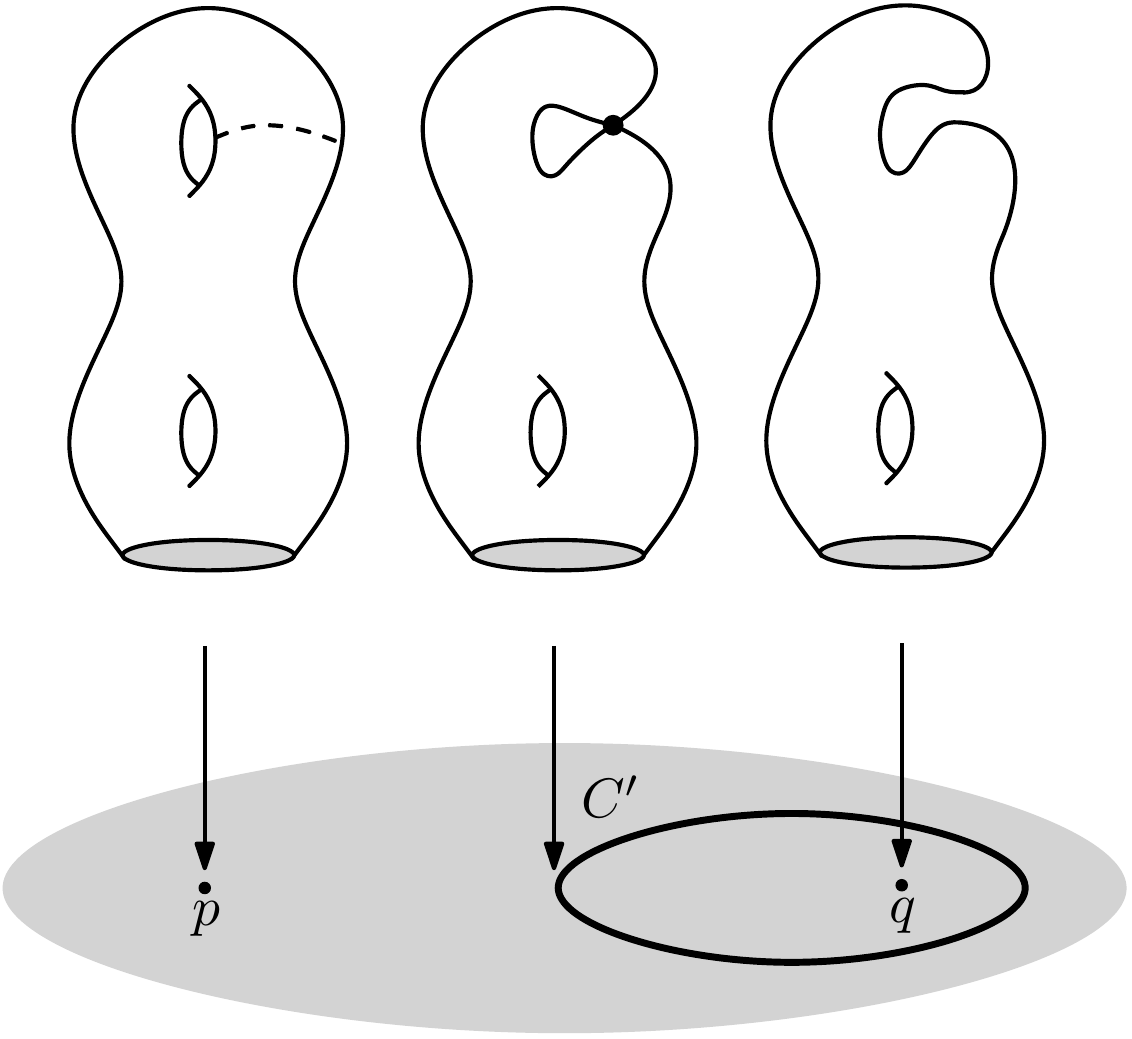}
  \caption{Passing a round 1-handle singularity}
  \label{fig:BrokenCriticalPoint}
\end{figure}

Now suppose that $f: X \rightarrow D^2$ is a Lefschetz fibration, possibly achiral, possibly broken.  Let $K$ be a framed knot in $f^{-1}(\partial D^2) \subset \partial X$, which can be isotoped so that it lies entirely on the interior of a single fiber.  Then we can attach a 2-handle along $K$ to yield a new manifold with boundary which we denote $X'$.  If we chose the framing along $K$ so that it is one less than the induced fiber framing, then $f$ will extend to a fibration on $X'$ with a new Lefschetz critical point in the newly added 2-handle.  If we instead choose $K$ to have framing one greater than the induced fiber framing, $f$ will instead extend to a fibration on $X'$ with an additional anti-Lefschetz critical point.

Suppose again that $f : X \rightarrow D^2$ is a fibration as above, but that we have now chosen two disjoint knots $K_1$ and $K_2$ in $\partial X$, each of which give a section of $f$ restricted to $f^{-1}(\partial D^2) \subset \partial X$.  Then we obtain a new manifold $X''$ by attaching $S^1 \times D^1 \times D^2$ to $\partial X$ along $K_1$ and $K_2$, by identifying $S^1 \times \{-1\} \times D^2$ and $S^1 \times \{1\} \times D^2$ with tubular neighborhoods of $K_1$ and $K_2$ respectively.  In this case the fibration $f$ will extend to $X''$, with a single round 1-handle singularity along $S^1 \times \{0\} \times \{0\}$.  Indeed, the knots $K_1$ and $K_2$ intersect each of the boundary fibers in a pair of points, which specify the locations of the 0-surgeries that take place as we pass the round 1-handle image.  Note that this also explains the choice of name for critical points of this type, as $S^1 \times D^1 \times D^2$ can be thought of as an $S^1$-family of 3-dimensional 1-handles $D^1 \times D^2$, which are attached to $X$ fiberwise along the boundary.  Alternatively, we can split $S^1 \times D^1 \times D^2$ into a 4-dimensional 1-handle and 2-handle pair, where the 2-handle runs over the 1-handle twice geometrically, but zero times algebraically.

The monodromy of the fibration outside this new round 1-handle singularity will depend on the framings of the tubular neighborhoods of $K_1$ and $K_2$, or alternatively, on the framing $k$ of the 2-handle in the 4-dimensional handle pair description.  Indeed, suppose that $F$ is the fiber of the fibration $f$ before attaching the round 1-handle, and that the monodromy around the boundary $\partial D^2$ is given by a map $\varphi:F \rightarrow F$.  Then adding the new round 1-handle changes the fibers along the boundary by replacing two disks $D_1$ and $D_2$ in $F$ with $S^1 \times [0,1]$.  The new monodromy will be given by the restriction of $\varphi$ to $F \backslash (D_1 \cup D_2)$, with $k$ Dehn twists along the cycle $S^1 \times \{\frac{1}{2}\}$.   

We will also sometimes refer to round 2-handles, which are the product of a 3-dimensional 2-handle with $S^1$.  These are, of course, just upside-down round 1-handles, and will not warrant any further discussion.

\section{Braided surfaces in $D^2 \times D^2$} 
\label{sec:BraidedSurfacesInD2xD2}

The construction of BLFs will be achieved by modifying certain coverings $h:X\rightarrow D^2\times D^2$, which are branched along surfaces in $D^2 \times D^2$.  To obtain a BLF, we will require that these branch loci are braided surfaces with caps in $D^2 \times D^2$, which we define here. 

\subsection{Braided Ribbon Surfaces}
\label{sec:braidedribbonsurfaces}
Rudolph defined a \emph{braided surface} \cite{Rudolph1983} to be a smooth properly embedded oriented surface $S \subset D^2 \times D^2$ on which the projection to the second factor $\mathrm{pr}_2:D^2 \times D^2 \rightarrow D^2$ restricts as a simple branched covering
.  Examples of these braided surfaces can be obtained by taking intersections of non-singular complex plane curves with 4-balls in $\mathbb{C}^2$, and they can be used to study the links that arise as their boundaries in $S^3 = \partial D^4$ (see e.g. \cite{Rudolph1985,Rudolph1993,Rudolph2005}).  

Let $S$ be a braided surface.  In a neighborhood of any branch point $p$ of the covering $\mathrm{pr}_2|_S$, there are local complex coordinates $u$ and $v$ on $D^2$ such that $S$ is given by the equation $u^2=v$, in the coordinates $(u,v)$ on $D^2 \times D^2$.  We say that $p$ is a \emph{positive} branch point if these coordinates can be taken to be orientation preserving, and a \emph{negative} branch point otherwise.   


One feature of Rudolph's braided surfaces are that they are all necessarily \emph{ribbon}.  A properly embedded surface $S$ in $D^4 = \{(z,w) : |z|^2+|w|^2 \leq 1\}$ is said to be \emph{ribbon embedded} if the function $|z|^2+|w|^2$ restricts to $S$ as a Morse function with no local maximal points on $\text{int }S$.  A properly embedded surface in $D^4$ is said to be \emph{ribbon} if it is isotopic to a surface which is ribbon embedded.  By fixing an identification of $D^2 \times D^2$ with $D^4$, we can similarly consider ribbon surfaces in $D^2 \times D^2$ (the definition of ribbon embeddings in $D^2 \times D^2$ will depend on our choice of identification, though the resulting class of ribbon surfaces will not).  

Rudolph proved that any orientable ribbon surface in $D^2 \times D^2$ is isotopic to a braided surface, though in general this isotopy cannot be chosen to fix $\partial S$, even if $\partial S$ is already a closed braid.


\subsection{Braided surfaces with caps}
\label{sec:BraidedSurfacesWithCaps}
The branch loci in $D^2 \times D^2$ we consider in this paper will not in general be ribbon, and hence cannot be braided via Rudolph's algorithm.  We thus consider a less restrictive notion of braiding, which we define now.

Let $\phi:F \rightarrow \Sigma$ be a smooth map of oriented surfaces.  Then a \emph{cap of $F$ with respect to} $\phi$ is an  embedded disk $D \subset F$, so that 
\begin{enumerate}
\item $\phi$ restricts to embeddings on $\text{int }D$ and on $\partial D$,
\item $F$ and $\Sigma$ both admit coordinate charts of the form $S^1 \times [-1,1]$ around $\partial D = S^1 \times \{0\}$ and $\phi(\partial D) = S^1 \times \{0\}$, on which $\phi$ is given by $(\theta,t) \mapsto (\theta,t^2)$,
\item in the above coordinate chart around $\phi(\partial D)$, the curve $S^1 \times \{1\}$ lies in $\phi(\text{int }D)$.  
\end{enumerate}

Now let $S \subset D^2 \times D^2$, and let $\mathrm{pr}_S$ denote the restriction of $\mathrm{pr}_2$ to $S$.  We say that $S$ is a \emph{braided surface with caps} if the critical points of $\mathrm{pr}_S$ all correspond either to isolated simple branch points or boundaries of caps of $S$ with respect to $\mathrm{pr}_S$.  Moreover, we will often assume that the critical values in $D^2$ form a set of embedded concentric circles (corresponding to the boundaries of caps), with isolated critical values  lying inside the innermost circle.  See Figure \ref{fig:braidsurfacewithcaps} for a cross sectional diagram of a braided surface with a single cap.  

\begin{figure}
 \centering
 \includegraphics[width=7cm]{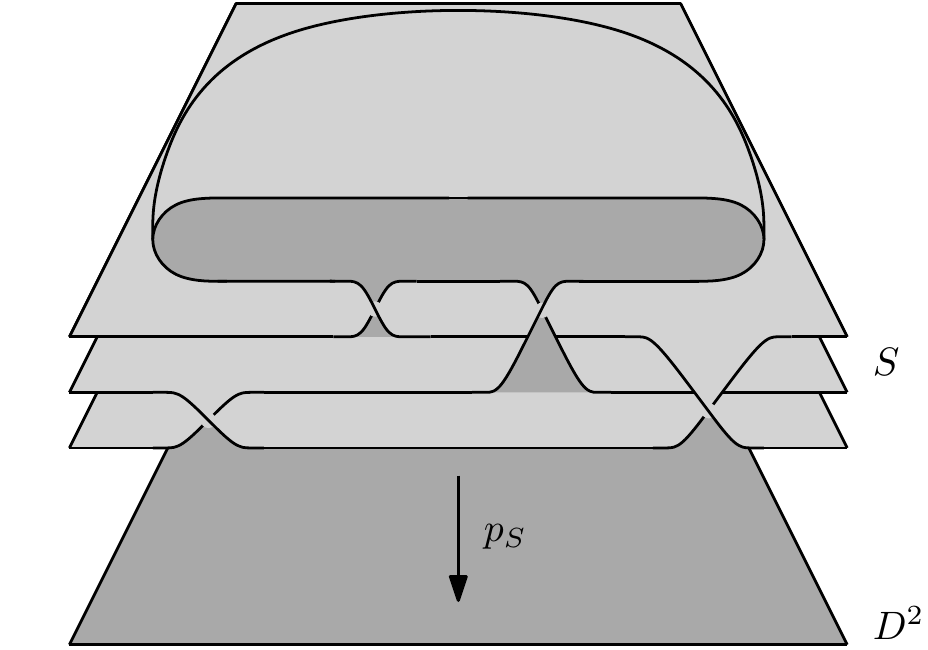}
  \caption{Cross section of a braided surface with caps}
  \label{fig:braidsurfacewithcaps}
\end{figure}

The following was proven in \cite{Hughes2015}.

\begin{theorem}
\label{thm:braidedsurfaceswithcaps}
Let $S$ be a smooth oriented properly embedded surface in $D^2 \times D^2$.  Then $S$ is isotopic to a braided surface with caps.  If $\partial S$ is already a closed braid, then the isotopy can be chosen rel $\partial S$.
\end{theorem}

\section{Broken Lefschetz fibrations from branched coverings}
\label{sec:BLFsBranchedCoverings}

Our method for creating BLFs is based on Proposition~\ref{prop:BranchedCoveringToBALF}, which takes as input a simple branched covering $h:X\rightarrow D^2 \times D^2$ with orientable branch locus, and yields a broken achiral Lefschetz fibration (BALF) $g:X \rightarrow D^2$.  By applying a local modification to our branched covering near the anti-Lefschetz critical points, we can remove all such critical points to yield a BLF $f:X\rightarrow D^2$.  Finally, we show how this technique can be combined with techniques of Gay and Kirby to produce BLFs over $S^2$ on many closed 4-manifolds.

\subsection{BALFs from branched coverings}
\label{sec:BALFsFromBranchedCoverings}

Proposition~\ref{prop:BranchedCoveringToBALF} is a straightforward generalization of Proposition~1.2 of \cite{LoiandPiergallini} to branched coverings with  non-ribbon branch loci.

\begin{prop}
\label{prop:BranchedCoveringToBALF}
Suppose that $X$ is a smooth 4-manifold with boundary, and that $h : X \rightarrow D^2 \times D^2$ is a simple branched covering with branch locus $B_h \subset D^2 \times D^2$ an embedded orientable surface.  Then there is an isotopy $\phi_t:D^2 \times D^2 \rightarrow D^2 \times D^2$, $\phi_0=\mathrm{id}_{D^2 \times D^2}$, such that $\mathrm{pr}_2 \circ \phi_1 \circ h : X \rightarrow D^2$ is a broken achiral Lefschetz fibration.
\end{prop}

\begin{proof}

By Theorem~\ref{thm:braidedsurfaceswithcaps}, $B_h$ is isotopic in $D^2 \times D^2$ to a braided surface with caps.  Let $\phi_t$ be an isotopy of $D^2 \times D^2$ which takes $B_h$ to such a surface.  Let $H = \phi_1 \circ h$ denote the isotoped branched covering, and let $B_H$ denote its branch locus.  Away from the preimages of the critical points of $\mathrm{pr}_2|_{B_H}$, the composition $g=\mathrm{pr}_2 \circ H$ is a regular map.  By \cite{LoiandPiergallini} $g$ has a Lefschetz (respectively anti-Lefschetz) critical point for every positive (respectively negative) branch point of $\mathrm{pr}_2|_{B_H}$.  

To see that the fold lines of $B_H$ along the boundaries of the caps give round 1-handle singularities, note that along these fold lines $B_H$ is locally embedded as $\mathbb{R}^2 \rightarrow \mathbb{R}^2 \times \mathbb{R}^2$, by $(s,r) \mapsto (0,r,s,r^2)$.  Furthermore, near nonsingular points of $B_H$, $H$ can be written in complex coordinates as $(u,v) \mapsto (u^2, v)$, where $B_H$ is given locally by $u=0$.  Combining these two local models yields a map of the required local form.  
Furthermore, the folds of $B_H$ can be pushed out so that they lie above a neighborhood of the boundary of $D^2$, so that their images form a collection of concentric circles in $D^2$ which enclose the Lefschetz and anti-Lefschetz critical values.  
\end{proof}

\begin{remark}
Note that Proposition~\ref{prop:BranchedCoveringToBALF} holds more generally than stated above.  Indeed, by \cite{Baykur,Lekili} any generic map $X \rightarrow D^2$ can be perturbed to become a BLF.  The proof of Proposition~\ref{prop:BranchedCoveringToBALF} is what will be most useful to us, since the branched covering $h : X \rightarrow D^2 \times D^2$ and the isotopy $\phi_t :D^2 \times D^2 \rightarrow D^2 \times D^2$ can often be constructed by hand from a given Kirby diagram of $X$ (see Section~\ref{sec:Examples}).
\end{remark}

\subsection{Replacing anti-Lefschetz critical points} Now suppose that $h:X \rightarrow D^2 \times D^2$ is a branched covering with branch locus $B_h \subset D^2 \times D^2$ an embedded orientable surface, such that $\mathrm{pr}_2 \circ h : X \rightarrow D^2$ is a broken achiral Lefschetz fibration.  It remains to show that after making local modifications to $h$ we can remove all anti-Lefschetz critical points.


\begin{prop}
\label{prop:antiLefschetzReplacement}
Let $h:X \rightarrow D^2 \times D^2$ be as above.  Then there is a branched covering $g:X \rightarrow D^2 \times D^2$ which agrees with $h$ outside of a small neighborhood of the anti-Lefschetz critical points of $\mathrm{pr}_2 \circ h$, such that $\mathrm{pr}_2 \circ g:X \rightarrow D^2$ is a broken Lefschetz fibration.
\end{prop}

\begin{proof}
We will show that Lekili's ``wrinkling" replacement move \cite{Lekili}, which replaces a single anti-Lefschetz critical point with three new Lefschetz critical points and a new broken singularity, can be lifted to our branch covering $h$.  

Let $p \in X$ be an anti-Lefschetz critical point of $f=\mathrm{pr}_2 \circ h$, and choose small disks $D, D'$ (thought of as sitting in the first and second factor of $D^2 \times D^2$ respectively), such that $D \times D'$ is a small neighborhood of $h(p) \in D^2 \times D^2$.  Note that $h(p)$ will lie on $B_h$, the branch locus of $h$, and will be a negative branch point of the map $\mathrm{pr}_2|_{B_h}:B_h \rightarrow D^2$.  The disks $D$ and $D'$ can be chosen so that $B_h \cap (D \times D')$ is given by $z^2=w$ for some orientation-reversing complex coordinates $(z,w)$ on $D \times D'$ centered at $h(p)$.  Hence $B_h \cap \partial (D \times D')$ will be a closed braid in $D \times \partial D'$ of index 2 with a single negative twist. 

If $n$ is a branched covering of degree $n$, then the preimage $h^{-1}(D \times D')$ will consist of $n-1$ components, $n-2$ of which are mapped homeomorphically onto $D \times D'$, while the remaining is mapped as a 2-to-1 covering of $D \times D'$ branched along $(D \times D') \cap B_h$.  Denote this latter component by $V$.

Pick a point $q \in D' \backslash f(p)$, and parametrized $D' \backslash q$ as $\partial D' \times (0,1]$.  Then for each $t \in (0,1]$ the intersection of $B_h$ with the solid torus $C_t = D \times (\partial D' \times t)$ will be a (possibly singular) closed braid.  Figure~\ref{fig:AntiLefCP} shows a sequence of braids whose closures are $B_h \cap C_t$ for various values of $t$.

\begin{figure}
 \centering
 \includegraphics[width=0.45\textwidth]{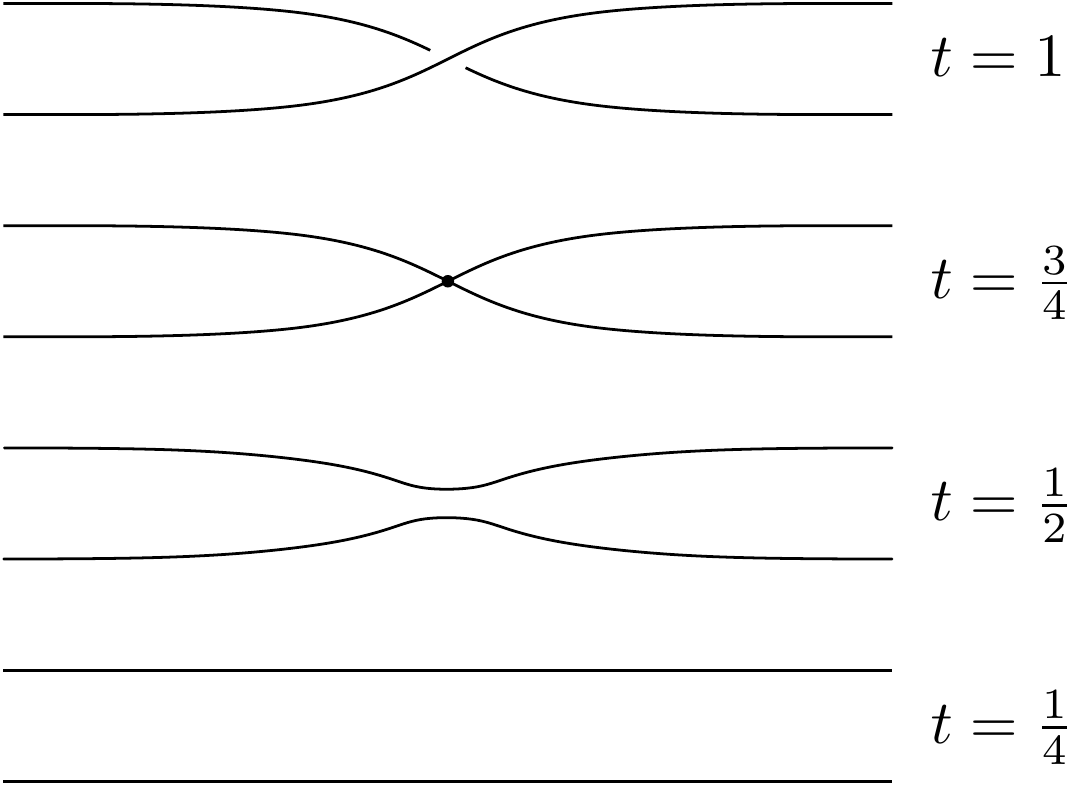}
  \caption{Branch locus corresponding to anti-Lefschetz critical point}
  \label{fig:AntiLefCP}
\end{figure} 

Consider instead the 2-fold branched covering $g:V' \rightarrow D\times D'$, whose branch locus in $D \times D'$ is given by the sequence of braids in Figure~\ref{fig:LekiliReplacement}, where at time $t=\tfrac{7}{8}$ a pair of strands is being born along the boundary of a cap.  The composition $\mathrm{pr}_2 \circ g : V' \rightarrow D'$ is a broken Lefschetz fibration, with three Lefschetz critical points contained inside one round 1-handle singularity.  By comparing the resulting monodromy with the monodromy from Lekili's wrinkle replacement \cite{Lekili}, we see that the corresponding fibrations are the same, and hence $V \cong V'$.  Moreover, the restriction of $g$ and $h$ to $\partial (D \times D')$ agree.  We can thus modify the branched covering $h$ in a neighborhood of any anti-Lefschetz critical points to obtain the desired broken Lefschetz fibration.
\begin{figure}
 \centering
 \includegraphics[width=\textwidth]{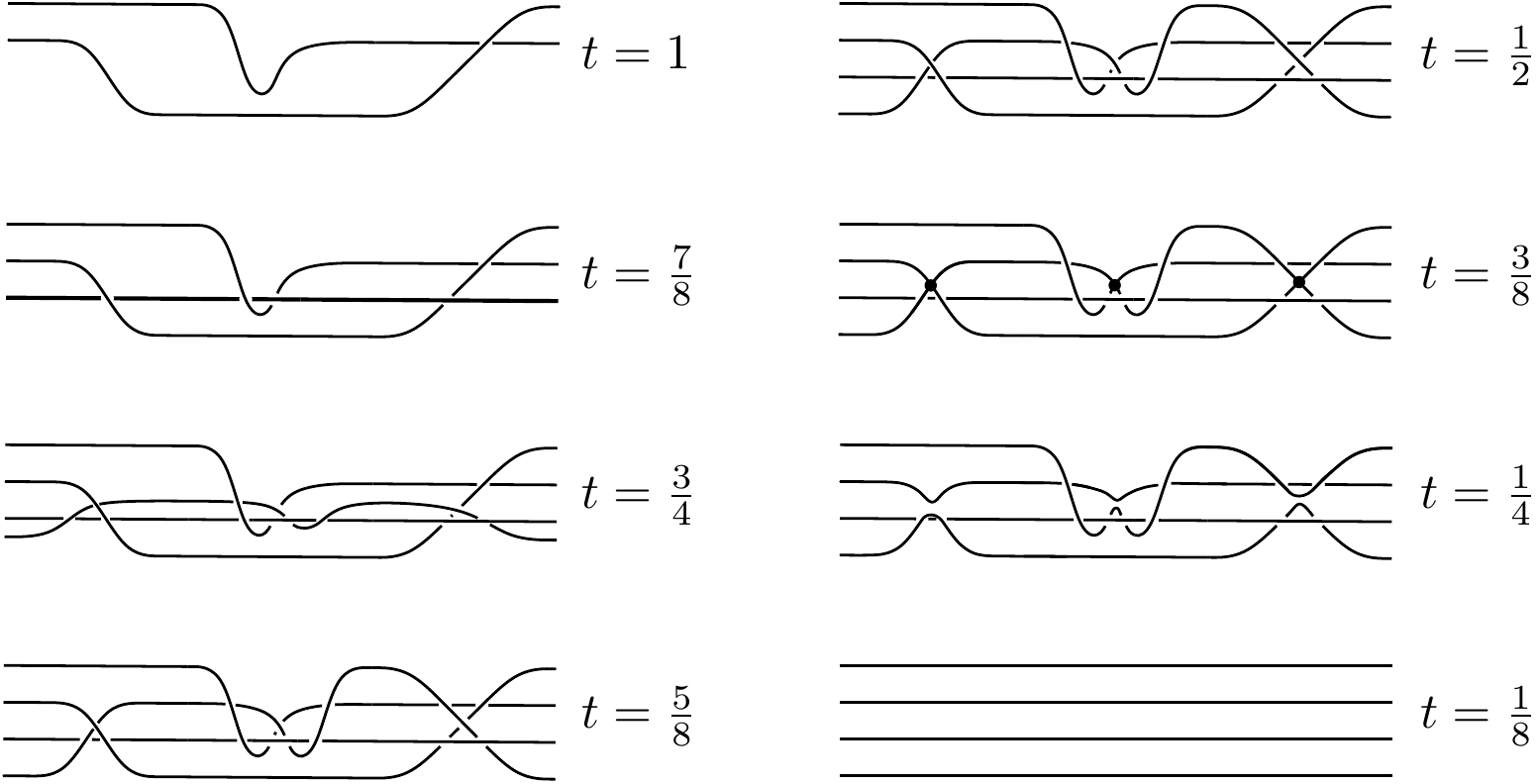}
  \caption{Branch locus corresponding to anti-Lefschetz replacement}
  \label{fig:LekiliReplacement}
\end{figure}
\end{proof}

\section{BLFs on closed 4-manifolds}  We now show how Propositions~\ref{prop:BranchedCoveringToBALF} and \ref{prop:antiLefschetzReplacement} can be used in many cases to construct a BLF $f:X \rightarrow S^2$ on a closed orientable 4-manifold $X$ from a given handle decomposition.  Let $F \subset X$ be a closed surface with $F \cdot F = 0$, and consider a tubular neighborhood $\nu F$ of $F$.  For simplicity, we describe first the construction in the case that $F \cong S^2$, and hence $\nu F \cong S^2 \times D^2$.  Such a neighborhood can sometimes be identified in the handle diagram of $X$ as a 2-handle attached along a 0-framed unknot together with the 0-handle of $X$.  If no such $S^2 \times D^2$ can be identified, it can be added to the diagram by adjoining a cancelling 2 and 3-handle pair, where the 2-handle is attached along a 0-framed unknot.  We will think of the union of this 2-handle with the 4-handle to which it is attached as forming $\nu F$.   

\subsection{Building the concave piece} \label{subsec:GayandKirby} We describe how to construct a concave broken fibration $f: \nu F \rightarrow S^2$ with a single round 1-handle singularity and no Lefschetz critical points.  This construction is originally due to Auroux, Donaldson, and Katzarkov \cite{AurouxDonaldsonKatzarkov}, as part of their construction of a broken Lefschetz fibration on $S^4$, though our description follows that in \cite{GayandKirby}.  

\begin{figure}
 \centering
 \includegraphics[width=\textwidth]{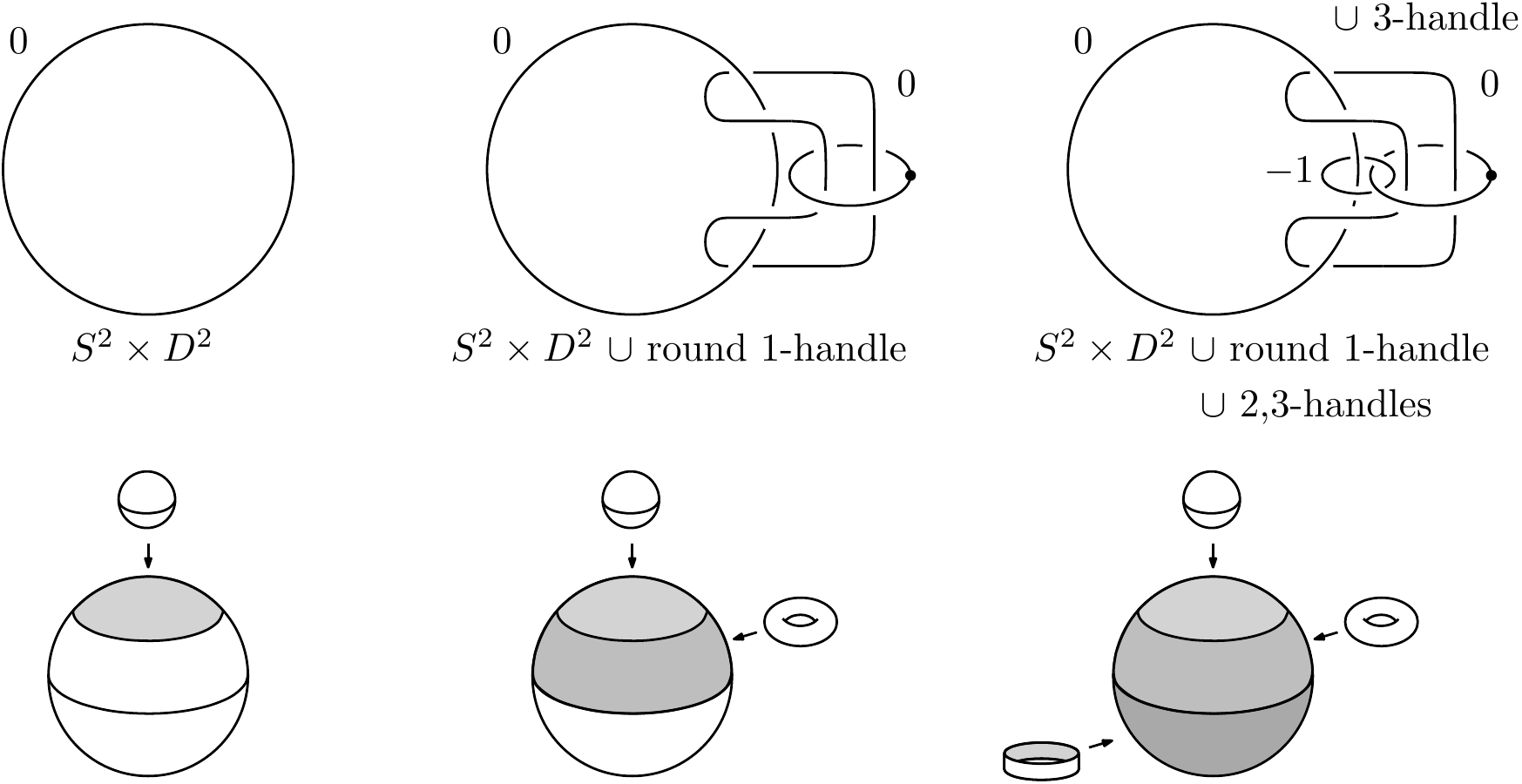}
  \caption{Concave broken fibration on $S^2 \times D^2$.}
  \label{fig:ADK}
\end{figure}

We begin by identifying the target of the projection $\mathrm{pr}_2 : S^2 \times D^2 \rightarrow
 D^2$ with the northern polar cap in $S^2$.  This defines a fibration of $S^2 \times D^2$ with fiber $S^2$ over this region (see the bottom left diagram in Figure~\ref{fig:ADK}).  Expressing $S^2 \times D^2$ with the usual handlebody diagram (top left Figure~\ref{fig:ADK}), we can add a 1-handle and 0-framed 2-handle to this diagram, as in the top middle diagram.  Taken together, these two handles can be interpreted as a round 1-handle, which is attached to $S^2 \times D^2$ along two sections of the existing fibration restricted to the boundary.  We can thus extend this fibration over the round 1-handle, giving a fibration over the northern hemisphere with a round 1-handle singularity over the arctic circle.  Fibers between the equator and arctic circle will be obtained from the polar fibers by 0-surgery, and hence will be tori.  Note that the fibration we have constructed so far is flat along its boundary.  
 
Finally, we add an additional 2-handle $H_2$, and a 3-handle $H_3$ to our diagram (top right, Figure~\ref{fig:ADK}).  The attaching circle of $H_2$ is a section of the flat fibration restricted to the boundary, and hence the fibration can be extended over $H_2$, by projecting it to the southern hemisphere (with fiber $D^2$).  The resulting fibration is concave.  The page of the boundary open book decomposition is a torus with a single hole (which resulted from attaching the 2-handle $H_2$), while its binding will be the belt-sphere of $H_2$.

The attaching sphere of the new 3-handle $H_3$ is arranged so that it intersects the binding at its north and south poles, and so that it intersects each page in a properly embedded arc.  The fibration can then be extended across $H_3$, resulting in no new critical points.  This extension changes the $D^2$ fibers over the southern hemisphere by adding a 2-dimensional 1-handle, yielding annular fibers.  On the other hand, the pages of the boundary open book change by the \emph{removal} of a neighborhood of a properly embedded arc from the puncture torus pages (the intersection of the original page with the attaching sphere of $H_3$), yielding annular pages.  

This gives a concave broken fibration as depicted in the bottom right diagram of Figure~\ref{fig:ADK}, with a single round 1-handle singularity, and no Lefschetz or anti-Lefschetz critical points.  Moreover, after sliding the 0-framed 2-handle off of the 1-handle, we find that the added 1,2, and 3-handles all form canceling pairs.  Hence the total space of our fibration is diffeomorphic to $S^2 \times D^2\cong \nu F$. Notice that the induced open book decomposition on $\partial (\nu F)$ will have disconnected binding, which may cause problems when we try to construct a matching convex fibration on $X \backslash \nu F$.  We thus instead think of the lone canceling 3-handle as being attached as a 1-handle to $X \backslash \nu F$, and construct a concave fibration $f_1$ on $X_1=\nu F \backslash \{3\text{-handle}\}$, whose boundary open book decomposition has punctured torus page and connected binding (see Figure~\ref{fig:HarerGenus}).

If instead $F$ has genus $g \geq 1$ we can proceed much as before, either identifying a neighborhood $\nu F$ in the handle diagram of $X$, or by adding a standard diagram of $F \times D^2$ with additional 2 and 3-handles to cancel the 1 and 2-handles of $\nu F$.  To this diagram we could add a round 1-handle and pair of 2 and 3-handles (see Figure~\ref{fig:HarerGenus}) and continue as above. 

\begin{figure}
 \centering
 \includegraphics[width=\textwidth]{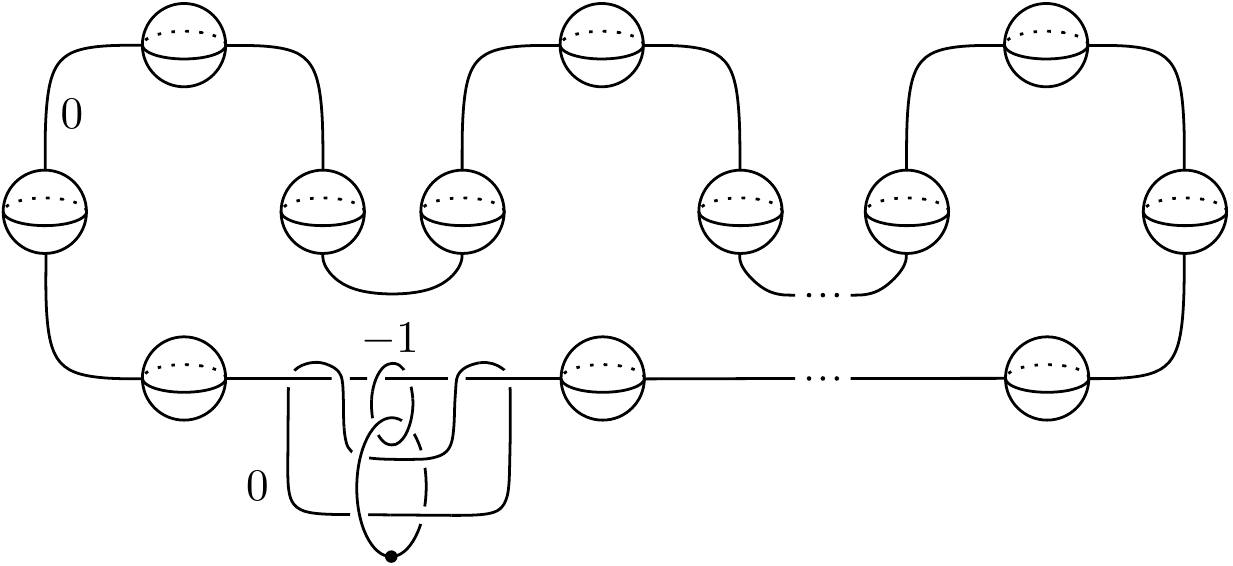}
  \caption[$F \times D^2$ with additional handles]{Neighborhood of $F \subset X$ with an extra 2-handle and round 1-handle added}
  \label{fig:HarerGenus}
\end{figure}


\subsection{Building the convex piece} 
\label{subsec:ConvexPiece} 
Let $Y=X\backslash X_1$.  We now discuss how to use a handle structure on $Y$ to build a convex fibration $g:Y \rightarrow D^2$, so that it extends the open book decomposition $\lambda:\partial Y \rightarrow D^2$ induced by the concave fibration $f_1:\nu F \rightarrow S^2$.  We attempt to do this in three steps:

\begin{enumerate}
\item Express the OBD $\lambda$ as $\lambda = \text{pr}_2 \circ h$, where $h : \partial Y \rightarrow \partial (D^2 \times D^2)$ is a simple covering branched along a closed braid in $\partial (D^2 \times D^2)$, and $\text{pr}_2 : D^2 \times D^2 \rightarrow D^2$ is the projection.
\item Extend the branched covering $h$ to a covering $H:Y \rightarrow D^2 \times D^2$ branched along an orientable surface.
\item Use Propositions~\ref{prop:BranchedCoveringToBALF} and \ref{prop:antiLefschetzReplacement} to obtain the desired BLF.
\end{enumerate}
Part (1) is always possible.  Indeed, let $P$ be the page of $\lambda$, with monodromy $\tau : P \rightarrow P$.  Then by choosing a suitable (degree $\geq$ 3 and simple) branched covering $\alpha:P\rightarrow D^2$, the map $\tau$ is the lift of a map $\widehat{\tau}:D^2 \rightarrow D^2$ which fixes the branch locus of $\alpha$ setwise \cite{Hilden,LabruereandParis}.  If $K$ is the binding of $\lambda$, then this allows us to write $\partial Y \backslash \nu K$ as a branched covering of the solid torus $D^2 \times \partial D^2$ branched over a closed braid.  A matching (unbranched) covering $\nu K \rightarrow \partial D^2 \times D^2$ can be glued to this covering to give the desired map $h :\partial Y\rightarrow \partial (D^2 \times D^2)$.

Problems may arise when we try to carry out Part (2) of the above process, however.  The covering $h$ can always be extended to a branched covering $H:Y \rightarrow D^2 \times D^2$, though the branch locus may not be an orientable embedded surface.

To see how this covering is constructed, fix some choice of relative handle decomposition for the pair $(Y, \partial Y)$.  The covering $h:\partial Y \rightarrow \partial (D^2 \times D^2)$ can be extended to a covering 
\[
\widehat{h} : \partial Y \times [0,1] \longrightarrow \partial (D^2 \times D^2) \times [0,1]
\]
in the usual way.  Here, $\partial Y \times [0,1]$ and $\partial (D^2 \times D^2) \times [0,1]$ are thought of as collar neighborhoods of $\partial Y$ and $\partial (D^2 \times D^2)$ respectively.  Identify $\partial Y$ with $\partial Y \times \{0\}$, and let $\partial_+ Y = \partial Y \times \{1\}$.  
We now attempt to extend this covering over the handles of $Y$ to construct the desired covering $H$.

Let $\sigma_1 = D^1 \times D^3$ be a 1-handle, and let $\tau_1 : \sigma_1 \rightarrow \sigma_1$ be the involution defined by 
\[
\tau_1: (t,x,y,z) \longmapsto (-t,-x,y,z).
\] 
If $\sigma_1$ is a 1-handle in our handle decomposition of $(Y,\partial Y)$, then we can isotope its attaching map $\alpha_1 : S^0 \times D^3 \rightarrow \partial_+ Y$ so that it is symmetric with respect to $\widehat{h}$, i.e. so that $\widehat{h}\circ \alpha_1 \circ \tau = \widehat{h} \circ \alpha_1$.  Once this is done, by Lemma 6.1 of \cite{BernsteinandEdmonds} we can extend $\widehat{h}$ over the 1-handle $\sigma_1$, using the quotient induced by $\tau_1$.  The result is a branched covering of $(\partial Y \times [0,1]) \cup \sigma_1$ over $\partial (D^2 \times D^2) \times [0,1]$, where the new branch locus is obtained by adding a disjoint disk to the branch locus of $\widehat{h}$ in $\partial (D^2 \times D^2) \times [0,1]$.

Similarly, if $\sigma_2 = D^2 \times D^2$ is instead a 2-handle attached to $\partial_+ Y$ by some attaching map $\alpha_2 : S^1 \times D^2 \rightarrow \partial_+ Y$, by \cite{Edmonds} we can isotope $\alpha_2$ so that it becomes symmetric with respect to the involution $\tau_2 : D^2 \times D^2 \rightarrow D^2 \times D^2$, defined by
\[
\tau_2: (t,s,x,y) \longmapsto (-t,s,-x,y).
\]
Here, the attaching circle of $\sigma_2$ will intersect the branching set of $\widehat{h}$ in two points, say $p_1$ and $p_2$.  Then we can extend the covering $\widehat{h}$ to a branched covering of $(Y \times [0,1])\cup \sigma_2$ over $\partial (D^2 \times D^2)$, where the new branch locus is obtained by attaching a single band to the branch locus of $\widehat{h}$ at the points $\widehat{h}(p_1)$ and $\widehat{h}(p_2)$.  This band will have $n$ half-twists in it, where $n$ is the framing of $\sigma_2$.

When extending $\widehat{h}$ over a 2-handle $\sigma_2$, it is possible that the corresponding band $\beta$ may be attached to the branch locus $B \subset \partial (D^2 \times D^2) \times [0,1]$ in a nonorientable way.  By \cite{Bobtcheva} this can be remedied, by adding (or removing) a half-twist in $\beta$ as in Figure~\ref{fig:OrientableMove}.  In this local picture we have pushed $B$ entirely into the 3-dimensional space $\partial (D^2 \times D^2) \times \{1\}$, where it can be depicted as an immersed surface with only ribbon double points.  The labels on the components denote the associated monodromy action on the sheets of $\widehat{h}$.

\begin{figure}
 \centering
 \includegraphics[width=0.8\textwidth]{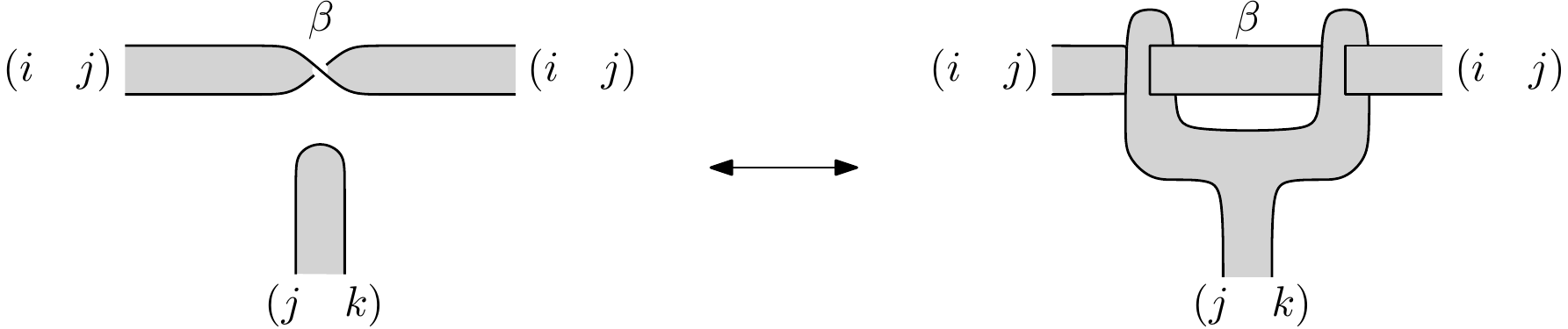}
  \caption[$F \times D^2$ with additional handles]{Fixing a nonorientable band in the branch locus of $\widehat{h}$}
  \label{fig:OrientableMove}
\end{figure}

Let $Y_2$ denote the union of $\partial Y \times [0,1]$ with the 1 and 2-handles.  We can thus extend the branched covering $h:\partial Y \rightarrow \partial (D^2 \times D^2)$ to a covering 
\[
\tilde{h}:Y_2\rightarrow \partial (D^2 \times D^2) \times [0,1]
\]
where the associated branch locus $\widetilde{B} \subset \partial (D^2 \times D^2) \times [0,1]$ is an embedded orientable surface.  If the intersection 
\[
\widetilde{B}_1=\widetilde{B}\cap\partial (D^2 \times D^2) \times \{1\}
\]
is an unlink, then $\tilde{h}$ can be extended across the 3 and 4-handles to give a branched covering $H:Y\rightarrow D^2\times D^2$ with orientable embedded branch locus.  This can be seen by noting that the union of the 3 and 4-handles is a thickened bouquet of circles, which can be expressed as a branched covering of $D^4$  with branch locus a collection of properly embedded disjoint disks.  If $\widetilde{B}_1$ is an unlink, this covering can be glued to $\tilde{h}:Y_2 \rightarrow \partial (D^2 \times D^2)\times [0,1]$ to obtain the desired covering $H$.  

In general however, $\widetilde{B}_1$ will not be an unlink.  By \cite{Piergallini} we can modify the covering by adding cusp and node singularities on the interior of $\widetilde{B}$ so that $\widetilde{B}_1$ becomes an unlink, though doing so may fail to preserve the required orientability of the branch locus $B$.  When this can be avoided, we can proceed with the rest of the construction to obtain a BLF of $X$ over $S^2$.

\subsection{BLFs on doubles of 4-manifolds}

We now discuss one situation in which the above construction will always be possible.  Let $U$ be a handlebody with single 0-handle and no 4-handles.  The \emph{double} of $U$ is the manifold $X  = U \cup_{\text{Id}_{\partial U}} \overline{U}$, where $\overline{U}$ denotes the handlebody $U$ with reversed orientation.  The handle structure on $U$ induces a handle structure on $X$ in a natural way, by turning the $j$-handles of $\overline{U}$ upside-down and attaching them as $(4-j)$-handles to $U$.

\begin{thm}
\label{thm:DoubleHandlebody}
Let $X$ be a smooth, closed, orientable 4-manifold, with handle structure coming from the double of a handlebody $U$.  Then the procedure described in Section~\ref{subsec:ConvexPiece} will produce a BLF $f:X\rightarrow S^2$.
\end{thm}

\begin{proof}

If $F=S^2$ is a trivially embedded sphere in the 0-handle of $X$, we can construct a concave fibration of $\nu F$ over $S^2$ as in Section~\ref{subsec:GayandKirby}. Let $Y = X \backslash \nu F$, and let $\lambda:\partial Y \rightarrow D^2$ be the induced open book decomposition.  By \cite{GayandKirby} the monodromy of $\lambda$ is trivial, and hence it factors through a simple branched covering $h : \partial Y \rightarrow \partial (D^2 \times D^2)$ of degree $\geq$ 3, whose branch locus is a trivial closed braid in $D^2 \times \partial D^2$. 

We now proceed to extend the covering $h$ to a covering 
\[
\tilde{h}:Y_2\rightarrow \partial (D^2 \times D^2)\times [0,1]
\]  
with branch locus $\widetilde{B}$.  Again we let $\widetilde{B}_1$ be the intersection of $\widetilde{B}$ with $\partial (D^2 \times D^2) \times \{1\}$.  Each 1-handle we extend over contributes an unknot component to $\widetilde{B}_1$ which is unlinked from the other components.

Before extending $h$ across the 2-handles, note that in the induced handle structure on $Y$, the 2-handles occur pairs.  Every 2-handle $\sigma$ from $U$ is paired with a 2-handle $\sigma'$ from $\overline{U}$, where $\sigma'$ is attached along a 0-framed meridian of the attaching circle of $\sigma$ (see \cite{GompfandStipsicz}).  We can also think of $\sigma'$ as being attached along the belt sphere of $\sigma$.  

Extending $h$ over a 2-handle from $U$ changes $\widetilde{B}_1$ by oriented surgery along a band $\beta$.  On the other hand, since the belt sphere of $\sigma$ is symmetric with respect to the involution $\tau_2: \sigma \rightarrow \sigma$, extending $h$ across $\sigma'$ will change $\widetilde{B}_1$ by oriented surgery along a band $\beta'$ which cancels $\beta$.  Hence the net effect of extending $h$ across $\sigma$ and $\sigma'$ does not change $\widetilde{B}_1$, which thus remains an unlink.   
\end{proof}

Any orientable $S^2$-bundle over a (possibly nonorientable) surface $\Sigma$ is the double of a $D^2$-bundle over $\Sigma$.  See \cite{GayandKirby} for an alternate construction of BLFs on doubles of 2-handlebodies.  

\subsection{Connected sums} The procedure outlined in Section~\ref{subsec:ConvexPiece} also respects connected sums in the following sense:

\begin{proposition}
Suppose that $X_1$ and $X_2$ are two handlebodies for which the procedure in Section~\ref{subsec:ConvexPiece} yields BLFs $f_1:X_1 \rightarrow S^2$ and $f_2: X_2 \rightarrow S^2$.  Then the same procedure can be used to obtain a BLF $f:X_1 \# X_2\rightarrow S^2$ which restricts to a concave fibration on $X_1 \backslash D^4 \subset X_1 \# X_2$ and to a convex fibration on $X_2 \backslash D^4 \subset X_1 \# X_2$.  Moreover, the ball $D^4 \subset X_1$ can be chosen so that $f|_{X_1\backslash D^4}= f_1|_{X_1\backslash D^4}$.
\end{proposition}

\begin{proof}
The handle structures on $X_1$ and $X_2$ yield a handle decomposition of $X_1 \# X_2$ by starting with the 0-handle of $X_1$ and attaching all 1,2 and 3-handles of $X_1$, followed by the 1, 2, 3 and 4-handles of $X_2$.

Cut out a neighborhood of an $S^2$ from $X_1$, and construct the concave fibration on $\nu S^2$ and the branched covering $h$ as above.  The map $h$ can be extended across the 1,2 and 3-handles of $X_1$ to give a covering 
\[
h' : X_1\backslash (\nu S^2 \cup 4\text{-handle}) \rightarrow \partial (D^2 \times D^2)\times [0,1].  
\]
We identify $\partial (D^2 \times D^2) \times [0,1]$ with $(D^2 \times D^2) \backslash (D'\times D')$, where $D' \subset D^2$ is a small disk containing the origin.  Then by \cite{Hughes2015} the branch locus $B'$ of $h'$ can be braided rel $\partial B'$ so that it is a braided surface with caps in $(D^2 \times D^2) \backslash (D'\times D')$.  Gluing the map $\text{pr}_2 \circ h'$ to the concave fibration on $\nu S^2$ gives a concave fibration $X_1 \backslash 4\text{-handle} \rightarrow S^2$.  This fibration can either be continued across the 4-handle of $X_1$ to obtain the fibration $f_1:X_1 \rightarrow S^2$, or across the 1, 2, 3 and 4-handles of $X_2$ to give a fibration $f:X_1 \# X_2 \rightarrow S^2$.
\end{proof}

\section{Examples}
\label{sec:Examples}

In this section we compute a few simple examples, to illustrate how the above procedure is carried out.

\subsection{BLF on $S^4$}
\begin{figure}
 \centering
 \includegraphics[width=0.3\textwidth]{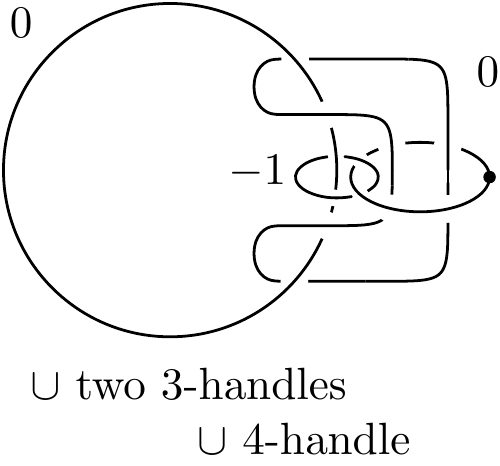}
  \caption[Handle structure on $S^2 \times D^2 \subset S^4$]{Handlebody structure of a neighborhood of $S^2$ in $S^4$}
  \label{fig:FourSphereDiagram}
\end{figure}
Consider the diagram of $S^4$ in Figure~\ref{fig:FourSphereDiagram}.  As in Figure~\ref{fig:ADK}, the union of all 0, 1, and 2-handles in this decomposition gives a neighborhood of an unknotted $S^2 \subset S^4$, together with an additional round 1-handle and (ordinary) 2-handle attached.  Call the union of these handles $X_1$, and set $X_2 = S^4 \backslash X_1$.  The open book decomposition on $\partial X_1 = \partial X_2$ induced by the concave fibration $f_1 : X_1 \rightarrow S^2$ from the above proof will have a punctured torus page with trivial monodromy (see \cite{GayandKirby}).  Hence it can be represented by a 3-fold simple branched covering $h:\partial X_2 \rightarrow \partial (D^2 \times D^2)$, and whose branch locus in $\partial (D^2 \times D^2)$ is the closure of the trivial 4-strand braid in $D^2 \times \partial D^2$ ($h$ can be described on each page by the branched covering in Figure~\ref{fig:3foldcover}).  

The branched covering $h$ extends to a covering $H:X_2 \rightarrow D^4$, which is built by turning the handle decomposition from Figure~\ref{fig:FourSphereDiagram} upside-down, and viewing $X_2$ as a 0-handle with two 1-handles attached.  The 0-handle can be expressed as a 3-fold covering of $D^4$ branched over two properly embedded unknotted disks.  For each 1-handle we extend this covering over, a properly embedded unknotted disk is added to the branch locus.  Hence the branch locus $B_H$ of $H$ in $D^4 \cong D^2 \times D^2$ is isotopic to the braided surface $\{p_1,\ldots , p_4\} \times D^2$, for some collection of disjoint points $\{p_1,\ldots , p_4\} \subset D^2$.  The only critical points in the resulting broken Lefschetz fibration $f:S^4 \rightarrow S^2$ will thus lie along round 1-handle singularity in $X_1$, and we recover Auroux, Donaldson, and Katzarkov's example in \cite{AurouxDonaldsonKatzarkov}.

\begin{figure}
 \centering
 \includegraphics[width=\textwidth]{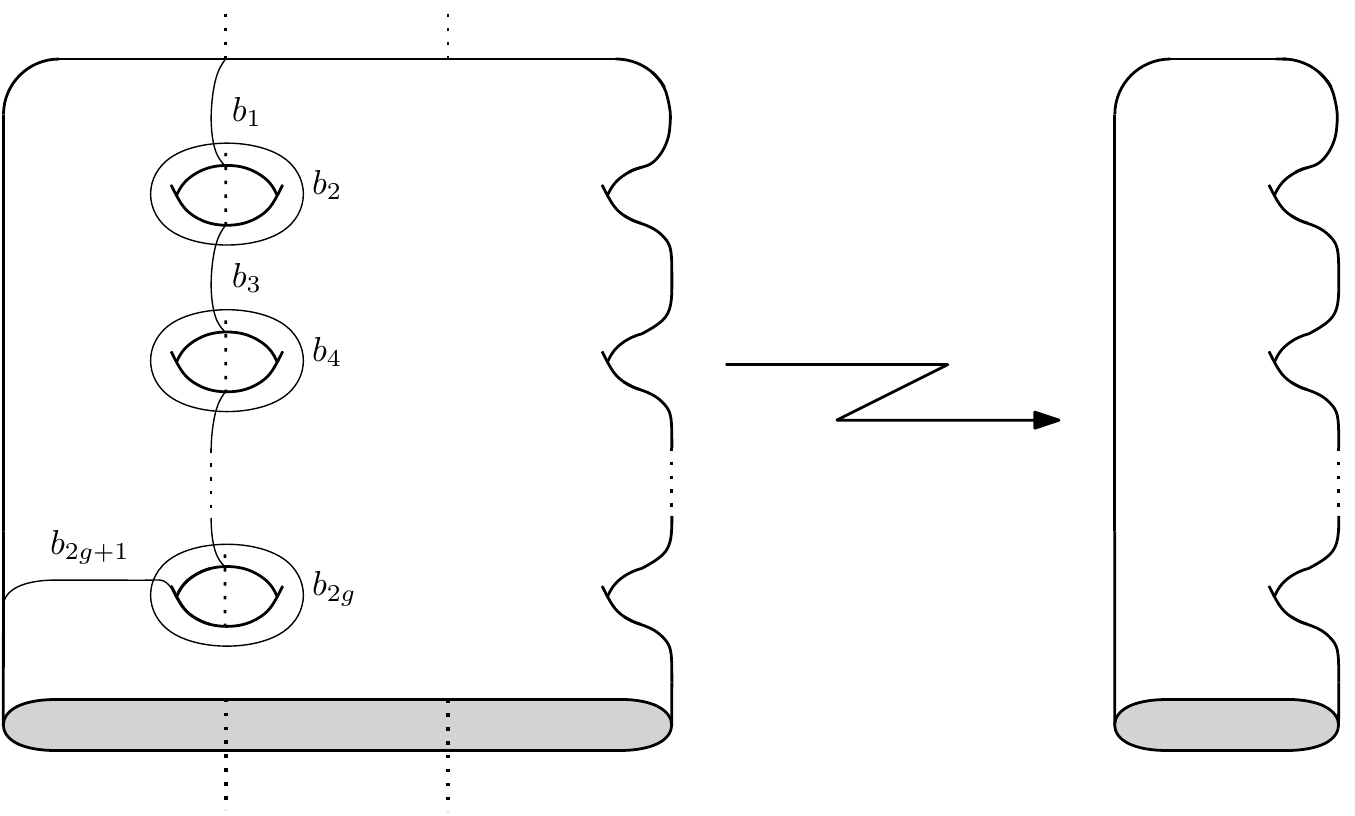}
  \caption{3-fold branched cover $\Sigma \rightarrow D^2$}
  \label{fig:3foldcover}
\end{figure}

\subsection{$S^2$-bundles over orientable surfaces}  Let $X$ be an $S^2$-bundle over an closed orientable surface of genus $g$.  For simplicity, we consider first the case when $g=1$.  Consider the diagram of $X$ in Figure~\ref{fig:S2bundle}, where each 1-handle attaching sphere is paired with the sphere directly across from it.  Notice that the diffeomorphism type of $X$ depends only on the parity of $n$.  Assume first that $n=0$.  In this case $X \cong S^2 \times T^2$.  While there is an obvious fibration $S^2 \times T^2 \rightarrow S^2$, the construction below has the advantage that it can be iterated to construct BLFs on connect sums of $S^2$-bundles, and generalizes to the twisted bundle $S^2 \tilde{\times} T^2$.

\begin{figure}
 \centering
 \includegraphics[width=0.8\textwidth]{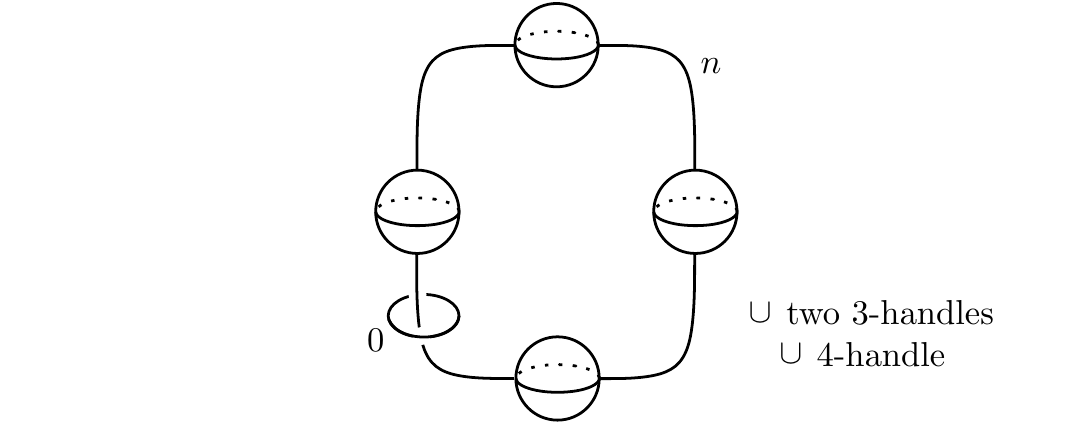}
  \caption{$S^2$-bundle over torus}
  \label{fig:S2bundle}
\end{figure}

We begin by adding a copy of the diagram in Figure~\ref{fig:FourSphereDiagram} (minus the 4-handle) to the diagram of $X$, which does not change the diffeomorphism type of $X$.  Again, let $X_1$ denote the union of the 0-handle with the newly added 1-handle and 2-handles, and let $X_2 =  X \backslash X_1$.  As above, $X_1$ admits a concave fibration over $S^2$ and induces an OBD on $\partial X_2$ with punctured torus page and trivial monodromy.  The associated 3-fold branched covering 
\[
h:\partial X_2 \rightarrow \partial (D^2 \times D^2)
\]
has branch locus a trivial 4-strand closed braid in $D^2 \times \partial  D^2$.  We need to extend $h$ over the handles in Figure~\ref{fig:S2bundle}, as well as the additional 3-handles we introduced when adding the diagram in Figure~\ref{fig:FourSphereDiagram}.  

In $\partial X_2$ there are four circles of branch points, corresponding to the four componends of the branch locus in $\partial (D^2 \times D^2)$.  We can isotope the handle attaching maps so that one of these four circles $C$ skewers the diagram in Figure~\ref{fig:S2bundle}, so that locally the covering looks like rotation of $\pi$ about the center of the diagram.  We first focus on extending the covering over the 1-handles $\sigma_1$ and $\sigma'_1$, and over the 2-handle $\sigma_2$ coming from the handle structure on $T^2$.

\begin{figure}
 \centering
 \includegraphics[width=0.38\textwidth]{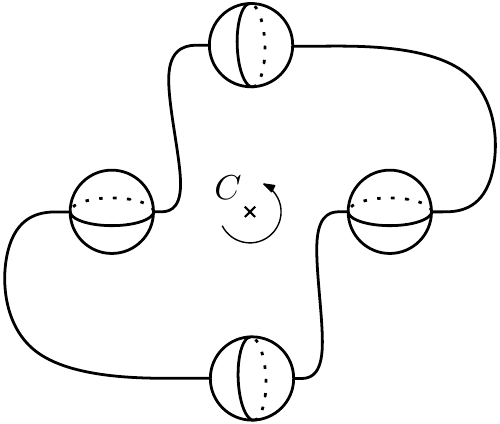}
  \caption{Symmetrizing the handles}
  \label{fig:RotationHandles}
\end{figure}

Isotope the attaching maps of these handles so that they are symmetric with respect to rotation by $\pi$ around $C$, as in Figure~\ref{fig:RotationHandles}.  We can thus extend the covering $\tilde{h}$ over $\sigma_1,\sigma'_1$, and $\sigma_2$.  Extending over the 1-handles adds a pair of disks to the branch locus, while extending over the 2-handle adds a band.  Notice that when the attaching circle of $\sigma_2$ runs along the horizontal 1-handle $\sigma_1$, it will intersect the branch set in precisely two points.  The branch $\widetilde{B}$ locus of
\[
\tilde{h}:(\partial X_2 \times [0,1])\cup \sigma_1\cup \sigma'_1\cup\sigma_2 \rightarrow \partial (D^2 \times D^2)\times [0,1]
\]
will be as in Figure~\ref{fig:FilmBranchLocus3}.

\begin{figure}
 \centering
 \includegraphics[width=0.75\textwidth]{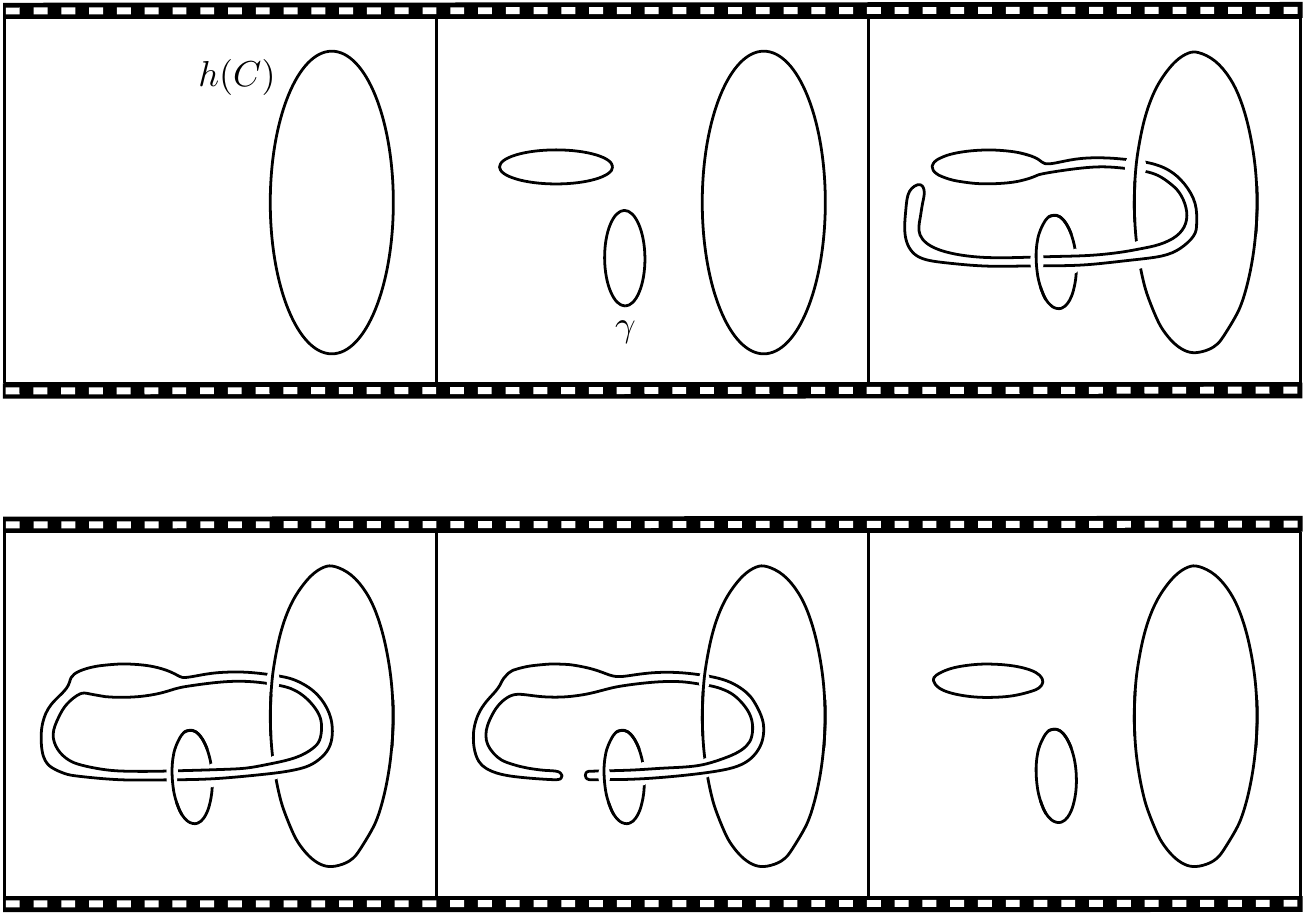}
  \caption{Branch locus $\widetilde{B}$ after extending over 1-handles and $\sigma_2$}
  \label{fig:FilmBranchLocus3}
\end{figure}

More precisely, let $\widetilde{B}_t = \widetilde{B} \cap ( \partial (D^2 \times D^2) \times \{t\})$ for $t \in [0,1]$.  The left-most frame represents $\widetilde{B}_0$, the branch locus of $h$, where we have suppressed all of the components except for $h(C)$.  As $t$ increases, we see two unknotted components appear, corresponding to the 1-handles $\sigma_1$ and $\sigma'_1$, followed by a band surgery corresponding to the 2-handle $\sigma_2$.  Extending $\tilde{h}$ across the remaining 2-handle in Figure~\ref{fig:S2bundle} results in an additional band surgery which cancels the first.  Note that all of the components in Figure~\ref{fig:FilmBranchLocus3} will have the same monodromy as $h(C)$.   

The branch locus $\widetilde{B}_1 = \widetilde{B} \cap (\partial (D^2 \times D^2) \times \{1\})$ is thus a six component unlink (three components from Figure~\ref{fig:FilmBranchLocus3} and three additional components from $h: \partial X_2 \rightarrow \partial (D^2 \times D^2)$ which were suppressed from the diagrams).  It only remains to extend this covering over the four 3-handles and unique 4-handle of $X_2$.  It is not hard to see that the union of these higher index handles admits a 3-fold simple branched covering over $D^4$, with branch locus consisting of six disjoint properly embedded disks in $D^4$.  This covering can thus be glued to $\tilde{h}$ to give a covering $H:X_2 \rightarrow D^2 \times D^2$, where these six disks cap off the six component unlink $\widetilde{B}_1$.  Let $B_H$ denote the branch locus of $H$, which consists of $\widetilde{B}$ capped off with these six disks.   

Finally, in order to apply Proposition~\ref{prop:BranchedCoveringToBALF} we must arrange $B_H$ as a braided surface with caps.  By \cite{Hughes2015} this is equivalent to arranging $B_H\subset D^2 \times D^2$ so that it sits in a collar neighborhood $\partial (D^2 \times D^2) \times [0,1]$ such that
\begin{enumerate}
\item the restriction to $B_H$ of the projection $\rho:\partial (D^2 \times D^2) \times [0,1]\rightarrow [0,1]$ is a Morse function, and 
\item $(\rho|_{B_H})^{-1}(t)$ a closed braid in $\partial(D^2 \times D^2) \times \{t\}$ for all regular values $t$.
\end{enumerate}

Figure~\ref{fig:BraidedFilmBranchLocus} shows how this can be done.  Again we start with the component $h(C)$ (hiding the three other components), and introduce two new unknots corresponding to extending the branched covering over the 1-handles.  The key difference now is that at every regular level the branch locus must be a closed braid.  Hence, the band corresponding to $\sigma_2$ now shows up first as a maximal point, which is then completed by adding two half-twisted bands via saddle points in the seventh frame.  The second band surgery takes place in the ninth frame.  Finally the branch locus is simplified to the trivial 3-strand braid, which is capped of by three minimal points (the other unseen three unknot components are similarly capped off).

\begin{figure}
 \centering
 \includegraphics[width=\textwidth]{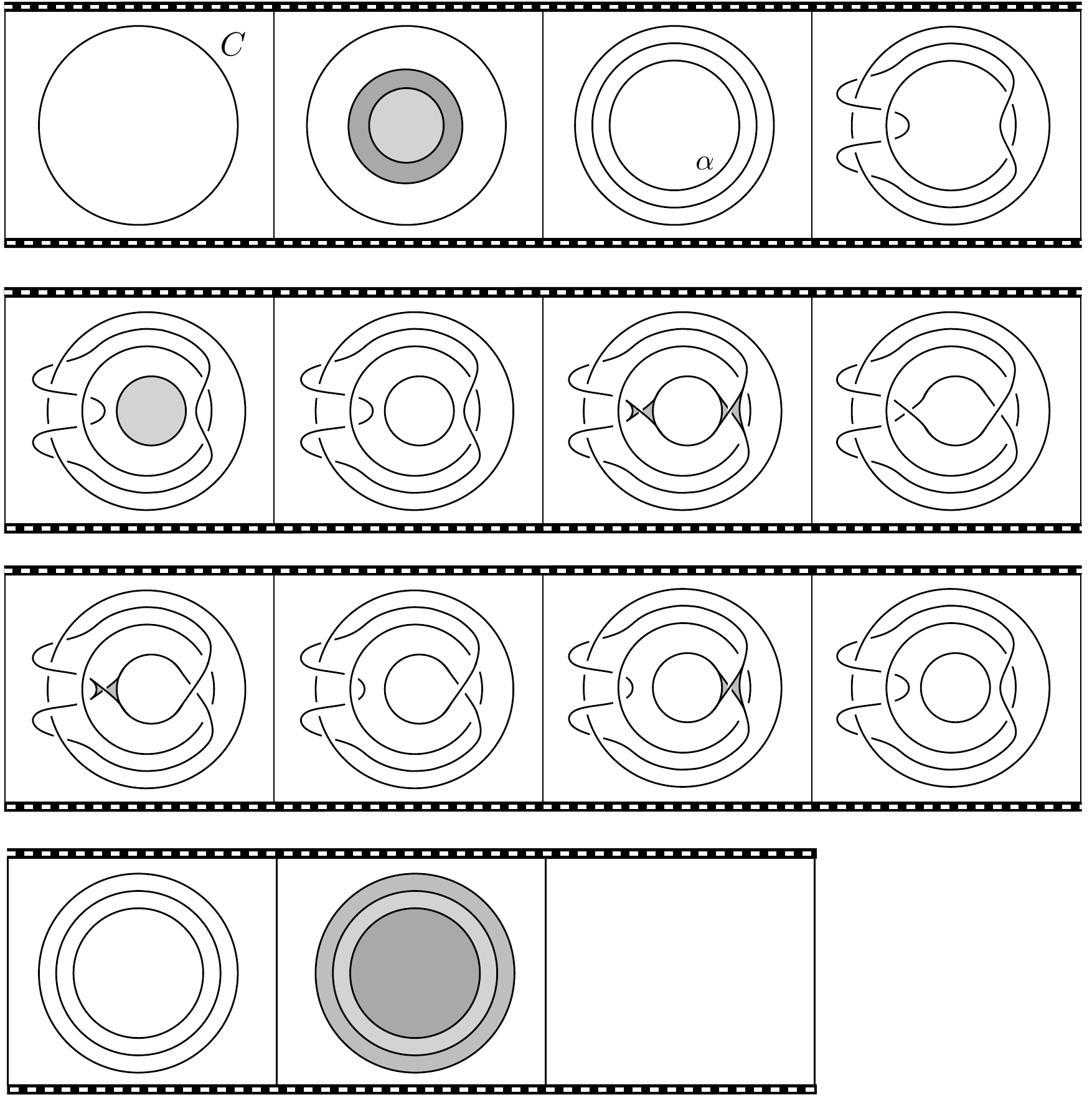}
  \caption{Branch locus $B_H$ as a braided surface with caps}
  \label{fig:BraidedFilmBranchLocus}
\end{figure}

The resulting fibration $\text{pr}_2 \circ H : X_2 \rightarrow D^2$ has a round 1-handle singularity for each maximal point of $B_H$ (which shows up along the boundary of the maximal disk), and a Lefschetz or anti-Lefschetz critical point for each saddle point.  Hence $\text{pr}_2 \circ H$ has three round 1-handle singularities, two Lefschetz critical points, and two anti-Lefschetz critical points.  The anti-Lefschetz critical points can replaced by Proposition~\ref{prop:antiLefschetzReplacement}, and the monodromy information of the fibration can be read off of Figure~\ref{fig:BraidedFilmBranchLocus}.

Now suppose that $X$ is the $S^2$-bundle over $T^2$ given by Figure~\ref{fig:S2bundle} with $n=1$, i.e. $X \cong S^2 \tilde{\times} T^2$.  Then the branch locus $\widetilde{B}$ will be as in Figure~\ref{fig:FilmBranchLocus3}, except that the band corresponding to $\sigma_2$ will have a single half-twist, and hence $\widetilde{B}$ will be nonorientable.  This can be remedied by involving another component of the branch locus $h:\partial X_2 \rightarrow \partial (D^2 \times D^2)$, and performing a move as in Figure~\ref{fig:OrientableMove} (see Figure~\ref{fig:FilmBranchLocus4}, where the monodromy information must be chosen to agree with the labels in Figure~\ref{fig:OrientableMove}).

\begin{figure}
 \centering
 \includegraphics[width=0.77\textwidth]{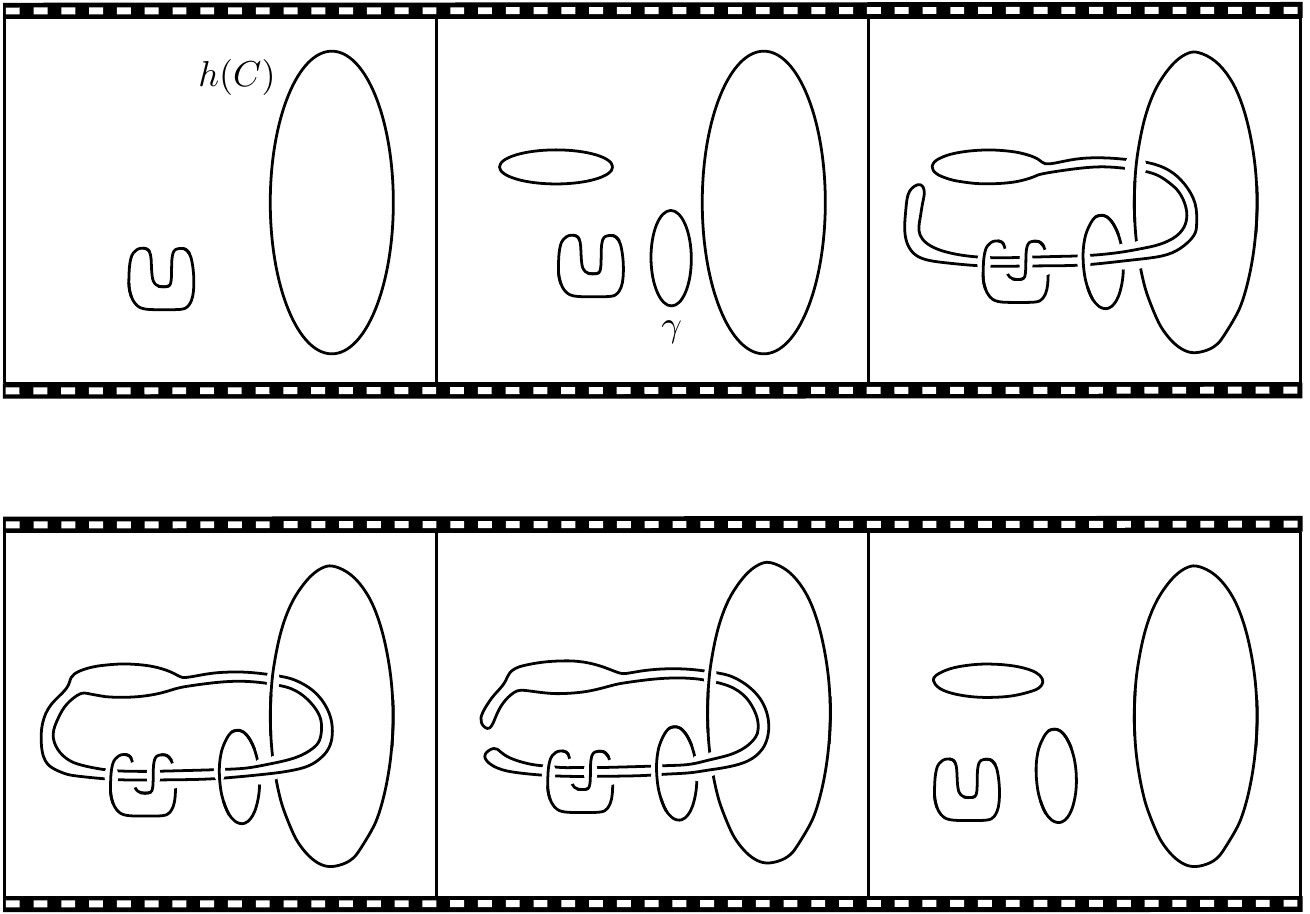}
  \caption{Branch locus $\widetilde{B}$ for $S^2 \tilde{\times} T^2$}
  \label{fig:FilmBranchLocus4}
\end{figure}

When braided, this move introduces a new local maximal point, and two new saddle points (one of each sign).  Hence the resulting broken fibration has an additional round 1-handle singularity, Lefschetz critical point, and anti-Lefschetz critical point when compared to the fibration constructed on $S^2 \times T^2$.

If $X$ is a $S^2$-bundle over a higher surface of genus $g>1$, we can start instead with the diagram in Figure~\ref{fig:S2bundleGenus}.  The associated branch locus will be as in Figure~\ref{fig:BraidedFilmBranchLocus}, except that the innermost strand $\alpha$ will be replaced by $2g-1$ parallel strands, and hence the fibration $\text{pr}_2 \circ H:X_2 \rightarrow D^2$ will now have $2g+1$ round 1-handle singularities.

\begin{figure}
 \centering
 \includegraphics[width=0.9\textwidth]{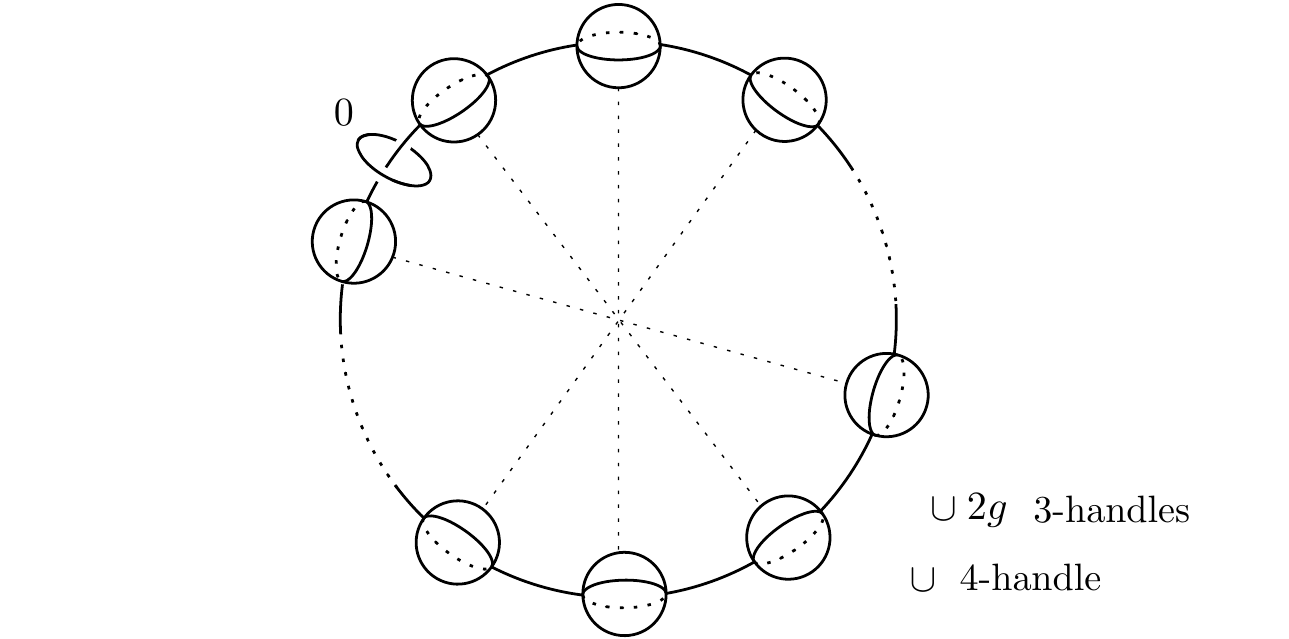}
  \caption{$S^2$-bundle over genus $g$ surface}
  \label{fig:S2bundleGenus}
\end{figure}

\subsection{$S^2$-bundles over $\mathbb{RP}^2$} We now consider $S^2$-bundles over $\mathbb{RP}^2$, which can be described by the diagram in Figure~\ref{fig:S2Nonorientable}.  Proceeding as above, we can arrange the component $C$ of the branch set so that it sits vertically in the diagram between the two strands of the attaching circle of the $n$-framed 2-handle $\sigma_2$, and so that the attaching maps of $\sigma_2$ and the 1-handle $\sigma_1$ are symmetric with respect to rotation about $C$.  For $n=0$ and $n=1$ the branch locus $\widetilde{B}$ will be as in Figures~\ref{fig:FilmBranchLocus3} and \ref{fig:FilmBranchLocus4} respectively, except that the second unknot components (labelled by $\gamma$ and corresponding to the extra 1-handle) will not be present.  After filling in the higher index handles and braiding the resulting branch locus $B_H$, the result will be the same as in Figure~\ref{fig:BraidedFilmBranchLocus}, except that in the second still only the outermost new component will appear.

\begin{figure}
 \centering
 \includegraphics[width=0.8\textwidth]{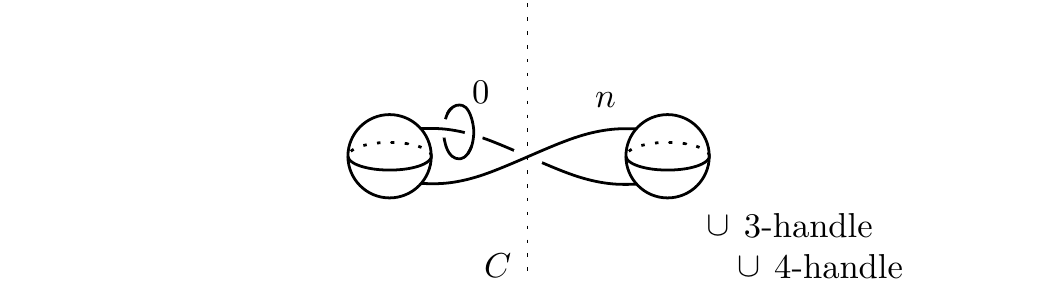}
  \caption{$S^2$-bundle over $\mathbb{RP}^2$}
  \label{fig:S2Nonorientable}
\end{figure}

\subsection{Connected sums} The above constructions can be repeated to give BLFs on connected sums.  For example, instead of capping off the unknot components in the third to last still of Figure~\ref{fig:BraidedFilmBranchLocus}, the movie (or another similar braided movie) could be repeated.

\bibliographystyle{plain}
\bibliography{bibliography}
\end{document}